\documentclass[a4paper, 11pt]{amsart}
\makeatletter
\@namedef{subjclassname@2020}{%
	\textup{2020} Mathematics Subject Classification}
\makeatother

\usepackage{amsmath,amsthm,verbatim,amssymb,amsfonts,mathrsfs,amscd, graphicx,enumerate,tikz-cd,stmaryrd}
\usetikzlibrary{trees}
\usepackage{multicol, caption}
\usepackage{fullpage}
\usepackage{enumitem}
\setlist[enumerate,1]{label={$(\roman*)$},leftmargin=*}

\usepackage{graphicx} 
\usepackage[OT2,T1]{fontenc}
\usepackage{amsmath}               
\usepackage{amsfonts}              
\usepackage{amsthm}                
\usepackage{mathtools}
\usepackage{verbatim}
\usepackage{xifthen}
\usepackage{amssymb}
\usepackage{xstring}
\usepackage{array}  
\usepackage{parskip} 
\usepackage{booktabs}
\usepackage[all]{xy}
\usepackage[pagebackref, hypertexnames=false]{hyperref} 
\hypersetup{
    colorlinks=true,
    linkcolor={red!50!black},
	citecolor={blue!50!black},
	urlcolor={blue!80!black}
}
\usepackage{tikz-cd}
\usepackage{stmaryrd} 
\usepackage{cleveref}
\usepackage{umoline}

\newtheorem{thm}{Theorem}[section]
\newtheorem{lem}[thm]{Lemma}
\newtheorem{prop}[thm]{Proposition}

\theoremstyle{definition}
\newtheorem{defn}[thm]{Definition}

\newtheorem*{exmp}{Example}
\newtheorem{remark}[thm]{Remark}

\newtheorem*{question*}{Question}

\numberwithin{thm}{section}
\numberwithin{equation}{section}

\crefname{table}{Table}{Tables}
\Crefformat{thm}{Theorem~$#2#1#3$}
\Crefformat{theorem}{Theorem~$#2#1#3$}
\Crefformat{definition}{Definition~$#2#1#3$}
\Crefformat{section}{Section~$#2#1#3$}
\Crefformat{prop}{Proposition~$#2#1#3$}
\Crefformat{corollary}{Corollary~$#2#1#3$}
\Crefmultiformat{corollary}{Corollaries~$#2#1#3$}{ and $#2#1#3$}{, $#2#1#3$}{, and $#2#1#3$}
\Crefformat{lem}{Lemma~$#2#1#3$}
\Crefformat{conjecture}{Conjecture~$#2#1#3$}
\Crefformat{table}{Table~$#2#1#3$}
\Crefformat{remark}{Remark~$#2#1#3$}
\Crefformat{algorithm-r}{Algorithm~$#2#1#3$}
\Crefmultiformat{algorithm-r}{Algorithms~$#2#1#3$}{ and $#2#1#3$}{, $#2#1#3$}{, and $#2#1#3$}

\usepackage[msc-links,alphabetic]{amsrefs}

\newcommand{\dotifempty} [1]{\ifthenelse{\isempty{#1}}
                          	{\cdot}%
                          	{#1}}

\DeclarePairedDelimiter\absolute{\lvert}{\rvert}%

\DeclarePairedDelimiter\braces{\{}{\}}

\newcommand{\abs}[1]{\absolute{\dotifempty{#1}}}
                         
\newcommand{\xmap}[0]{\xrightarrow}

\renewcommand{\restriction}[2]{{
  \left.\kern-\nulldelimiterspace 
  #1 
  \vphantom{\big|} 
  \right|_{#2} 
  }}

\graphicspath{ {images/} }

\usepackage{shortcuts}

\title{Maximal stable lattices in representations over discretely valued fields}
\author{Amit Ophir and Ariel Weiss}
\address{Amit Ophir, Department of Mathematics, University of California, San Diego.\vspace*{-12pt}}
\address{Department of Mathematics, The Weizmann Institute of Science.
\vspace*{-3pt}}
\email{amit.ophir@weizmann.ac.il}
\address{Ariel Weiss, Department of Mathematics, The Ohio State University.\vspace*{-12pt}}
\address{Department of Mathematics, Trinity College.\vspace*{-3pt}}
\email{ariel.weiss@trincoll.edu}
 \subjclass[2020]{20G25, (11F70 20G05, 51E24)}
 \keywords{Ribet's Lemma, representations over discretely valued fields, stable lattices, Bruhat--Tits buildings}

\begin{document}

\maketitle

\begin{abstract}
    Let $\rho\:G\to \GL_n(K)$ be an continuous irreducible representation of a compact group over a complete discretely valued field $K$.
    Let $W_i,W_j$ be two irreducible subrepresentations of $\overline{\rho}^{ss}$, the semisimplification of the residual representation.
    We study the structure of the $G$-stable lattices $\Lambda\sub K^n$ with a view to understanding the question of when $\rho$ realises a non-split extension of $W_i$ by $W_j$.
    
    In particular, we introduce the notion of a \emph{maximal} $G$-stable lattice and prove that any non-split extension of $W_i$ by $W_j$ that can be realised by $\rho$ can also be realised by a maximal lattice. As applications, we give a new proof and a strengthening of Bella\"iche's generalisation of Ribet's Lemma \cite{Bellaiche-apropos}, which assures the abundancy of non-split extensions that can be realised by $\rho$. On the other hand, we also show that, if the representations $W_i, W_j$ occur with multiplicity one in $\orho^{ss}$, then $\rho$ can realise \emph{at most} one non-split extension of $W_i$ by $W_j$.
\end{abstract}

\section{Introduction}

In his celebrated 1976 article \cite{Ribet-herbrand}, Ribet pioneered a method to construct interesting extensions of number fields by studying stable lattices inside suitably constructed representations.
A key component of his method is the so-called \emph{Ribet's Lemma}. Let $K$ be a complete discretely valued field with ring of integers $\O_K$, uniformiser $\pi$,  and residue field $\F$. Let $G$ be a compact group and let $\rho\:G\to \GL_2(K)$ be a continuous irreducible representation that is \emph{residually reducible}: there exist characters $\chi_1, \chi_2\:G\to \F\t$ such that the semisimplification of the residual representation $\orho^{ss}$ splits as $\orho^{ss} = \chi_1\+\chi_2$. Ribet's Lemma \cite{Ribet-herbrand}*{Prop.\ 2.1} asserts that there is a $\rho(G)$-stable lattice $\Lambda\sub K^2$ such that $\Lambda/\pi\Lambda$ is a non-split extension of $\chi_1$ by $\chi_2$. 

More generally, suppose that $\rho\:G\to \GL(V)$ is a finite-dimensional continuous irreducible representation over $K$, that $W_i, W_j$ are irreducible subrepresentations of $\orho^{ss}$, and that $U$ is a non-split extension of $W_i$ by $W_j$. We say that $\rho$ \emph{realises} $U$ if there is a $\rho(G)$-stable lattice $\Lambda\sub V$ such that $U$ is a subquotient of $\Lambda/\pi\Lambda$.

The set of homothety classes of $\rho(G)$-stable lattices, which we denote by $\X(\rho)$, can be naturally equipped with a simplicial complex structure by viewing it as a subcomplex of the Bruhat--Tits building for $\PGL(V)$. This perspective, first suggested by Serre, has been studied in \cites{Bellaiche-apropos, Bellaiche-Graftieaux,Bellaiche-arbres, ophirweiss, suh-1, suh-2, suh2023stable}. 

In this paper, we study the structure of $\X(\rho)$ for a fixed representation $\rho$, with a view to understanding the non-split extensions that $\rho$ realises. 
In particular, we explore the notion of a \emph{maximal stable lattice}, and prove that much of the structure of $\X(\rho)$ can be obtained from the maximal lattices. For example, we show that every subquotient of $\orho$ that realises a non-split extension of a $W_i$ by a $W_j$ can already be realised as a subrepresentation of $\Lambda/\pi\Lambda$ for some maximal stable lattice $\Lambda$ (\Cref{prop_sharp_lying_above} and \Cref{thm_maximal_sharp}).
This result is in the spirit of the version of Ribet's Lemma in \cite[Lemma 3]{mazur2011can}, and allows us to give a new proof of Bella\"iche's generalisation of Ribet's Lemma \cite{Bellaiche-apropos}. We also show that the geometric structure of $\X(\rho)$ can be completely recovered from the maximal lattices (\Cref{thm:trop-convex}), thereby strengthening a recent result of Suh \cite{suh2023stable}.

Maximal lattices can also be characterised as those lattices $\Lambda$ for which the socle $\soc(\Lambda/\pi\Lambda)$ of the $\F[G]$-module $\Lambda/\pi\Lambda$ is irreducible (\Cref{thm_maximal_equivalence}). Such lattices were previously studied in \cite{emerton2015lattices} in the case that $\orho^{ss}$ is multiplicity-free. Our results apply without this hypothesis, and recast this algebraic notion in the language of Bruhat--Tits buildings.

\subsection{Maximal stable lattices}

If $\Lambda$ is a $\rho(G)$-stable lattice and $v\in V$ is nonzero, then we can always replace $\Lambda$ with a homothetic lattice $\Lambda'$ such that $v\in \Lambda'\setminus \pi\Lambda'$: we say that $\Lambda'$ is \emph{normalised} at $v$. Let $\LL_v$ be the set of all $\rho(G)$-stable lattices, normalised at $v$, and equip $\LL_v$ with the partial order given by inclusion.

\begin{defn}\label{def_maximal_lattices}
    We say that $\Lambda\in\LL_v$ is maximal with respect to $v$ if it is maximal in $\LL_v$. We say that a $\rho(G)$-stable lattice $\Lambda$ is \emph{maximal} if it is maximal with respect to some nonzero $v\in V$.
\end{defn}

In \Cref{thm_maximal_equivalence}, we show that a $\rho(G)$-stable lattice $\Lambda$ is maximal if and only if $\soc(\Lambda/\pi\Lambda)$ is irreducible. Writing $\orho^{ss} = \bigoplus_{i=1}^rW_i^{m_i}$, where the $W_i$'s are distinct irreducible representations of $G$ over $\F$, it follows that if $\Lambda$ is a maximal stable lattice and $\soc(\Lambda/\pi\Lambda)\simeq W_i$, then $\Lambda/\pi\Lambda$ always contains a subrepresentation that is a non-split extension of $W_i$ by some $W_j$.

In \Cref{lem_uniqueness}, we prove that there are an abundance of maximal lattices: for each $i = 1, \ldots, r$, there is a maximal lattice $\Lambda_i$ such that $\soc(\Lambda_i/\pi\Lambda_i)\simeq W_i$. As a consequence, in \Cref{thm_ext_classes_by_maximal_lattices}, we deduce that if $U$ is a non-split extension that can be realised by $\rho$, then $U$ is isomorphic to a subrepresentation of $\Lambda/\pi\Lambda$ for some maximal lattice $\Lambda$. On the other hand, we prove that there aren't too many maximal lattices. For example if $m_i = 1$, then we show in \Cref{lem_uniqueness} that the maximal lattice $\Lambda_i$ is unique.

\subsection{Generalisations of Ribet's Lemma}

Ribet's Lemma has been generalised to higher-dimensional representations, to mod $\pi^n$ congruences, and to representations over more general rings \cites{Urban-ribet-lemma, Bellaiche-apropos,Bellaiche-Graftieaux,Brown, Bellaiche-book, Bellaiche-arbres, ophirweiss, dasgupta-ribet}. These generalisations are crucial components of proofs of Iwasawa main conjectures, the Brumer--Stark conjecture, and cases of the Bloch--Kato conjecture \cites{Wiles_Iwasawa, Urban,Bellaiche-book, Skinner-Urban, Dasgupta-BS, dasgupta-BS-2}. 

Let $\rho\:G\to \GL(V)$ be a continuous irreducible representation as before, and write $\orho^{ss} = \bigoplus_{i=1}^rW_i^{m_i}$. In \cite{Bellaiche-apropos}, Bella\"iche defines the following graph, which captures when $\rho$ realises an extension of $W_i$ by $W_j$.

\begin{defn}[Bella\"iche's graph of extensions]
    Let $\Gamma_\rho$ be the directed graph whose vertices are the irreducible representations $W_1, \ldots, W_r$, and where there is a directed edge from $W_j$ to $W_i$ if $\rho$ realises a non-split extension of $W_i$ by $W_j$, i.e.\ a nonzero element of  $\Ext_{\F[G]}(W_i, W_j)$.
\end{defn}

In Ribet's setting \cite{Ribet-herbrand}, $\dim V = r = 2$, and Ribet's Lemma states that there are extensions of $W_1$ by $W_2$ and of $W_2$ by $W_1$ that are realised by $\rho$. Equivalently, $\Gamma_\rho$ is connected as a directed graph. In \cite{Bellaiche-apropos}*{Thm.\ 1}, Bella\"iche proves that the graph $\Gamma_\rho$ is connected as a directed graph in general.

As a first application of our results on maximal vertices, we give a new proof of this result. As well as being simpler than Bella\"iche's proof, our methods are more precise in that they provide an upper bound on the directed distance between any two vertices $W_i$ and $W_j$. 

\begin{thm}\label{thm:bellaiche-generalisation}
    For each $i = 1, \ldots r$, there exists a maximal lattice $\Lambda_i$ with $\soc(\Lambda_i/\pi\Lambda_i)\simeq W_i$. For each $j = 1, \ldots r$, suppose that $W_j$ is in the $m$-th level of the socle filtration of $\Lambda_{i}/\pi\Lambda_{i}$. Then there is a directed path in $\Gamma_\rho$ from $W_i$ to $W_j$ of length $m-1$. In particular, $\Gamma_\rho$ is connected.
\end{thm}

\subsection{Uniqueness of extensions realised by $\rho$}

In light of Ribet's Lemma, which attempts to use $\rho$ to construct interesting non-split extensions, it is natural to ask whether a single $\rho$ can realise multiple distinct non-split extensions of $W_i$ by $W_j$. As an application of \Cref{thm_ext_classes_by_maximal_lattices}, we give a new proof that, if $W_i$ and $W_j$ both occur with multiplicity $1$ in $\orho^{ss}$, then $\rho$ can realise at most one non-split extension of $W_i$ by $W_j$.

\begin{thm}\label{thm:uniqueness}
    Let $1 \le i, j \le r$ and suppose that $m_i = m_j=1$, i.e.\ that $W_i$ and $W_j$ occur with multiplicity one in the decomposition of $\orho^{ss}$. Then, up to equivalence,\footnote{See \Cref{def:equivalence} for the notion of equivalence of extensions.} $\rho$ realises at most one non-split extension of $W_i$ by $W_j$.
\end{thm}

When $r=2$, this was proved by \cite{Bellaiche-book}*{Props.~1.74, 1.75}. More generally, when $\orho^{ss}$ is multiplicity-free, the result can be deduced from \cite{emerton2015lattices}*{Sec.~4}.

\subsection{The geometry of maximal lattices}


In \cite{suh2023stable}, Suh studies the geometry of the complex $\X(\rho)$ in terms of its extremal vertices, which, algebraically, correspond to $\rho(G)$-stable lattices $\Lambda$ such that $\Lambda/\pi\Lambda$ is indecomposable \cite{suh2023stable}*{Prop.\ 1.1.1}. In \cite{suh2023stable}*{Thm.\ 1.1.2}, Suh proves that $\X(\rho)$ is the simplicial convex hull of the set of extremal vertices, and proves that, if $\orho^{ss} \simeq\bigoplus_{i=1}^rW_i^{m_i}$, then the number of extremal vertices is bounded below by $\min\{3, \sum_{i=1}^rm_i\}$.

In \Cref{thm_maximal_equivalence}, we prove that $\Lambda$ is a maximal lattice if and only if $\soc(\Lambda/\pi\Lambda)$ is irreducible. In particular, if $\Lambda$ is maximal, then $\Lambda/\pi\Lambda$ is indecomposable, so the homothety class of $\Lambda$ is an extremal vertex of $\X(\rho)$. We use this observation in \Cref{sec:suh} to strengthen Suh's results. First, in \Cref{thm:trop-convex}, we show that $\X(\rho)$ is the tropical\footnote{See \Cref{def:tropical} for this notion. We expect that any set that is simplicially convex is also tropically convex, so that $\X(\rho)$ is also the simplicial convex hull of the set of maximal vertices, but we have not been able to find a reference for this fact in the literature.} convex hull of the set of maximal vertices, not just the set of extremal vertices, and we show that the set of maximal vertices is a minimal set whose tropical convex hull is $\X(\rho)$. Second, in \Cref{thm_maximal_vertices_lower_bound}, we show that the number of maximal vertices, and hence the number of extremal vertices, is bounded below by $\sum_{i=1}^rm_i$.

\section{Preliminaries}

Throughout this section, let $K$ be a complete discretely valued field, with ring of integers $\O_K$, uniformiser $\pi$ and residue field $\F=\O_K/\pi\O_K$. Let $V$ be a finite-dimensional vector space over $K$.

\subsection{The Bruhat--Tits building of $\PGL(V)$}

Recall that a lattice $\Lambda$ in $V$ is a free $\O_K$-submodule of $V$ of rank $\dim_K(V)$. 
Two lattices $\Lambda_1,\Lambda_2$ in $V$ are called \emph{homothetic} if there exists $\lambda\in K\t$ such that $\Lambda_2 = \lambda\Lambda_1$, and we denote by $\X(V)$ the set of homothety classes of lattices in $V$.
Given a homothety class $x\in \X(V)$, we denote by $\Lambda_x$ a choice of a lattice in $x$.

\begin{defn}[Bruhat--Tits building of $\PGL(V)$]
    The set $\X(V)$ is equipped with a structure of an abstract simplicial complex as follows:
\begin{itemize}[leftmargin=*]
    \item The vertices are the points of $\X(V)$. 
    \item A collection of $n+1$ vertices $x_0,\dots,x_n\in\X(V)$ forms an $n$-dimensional face if there exist representative lattices $\Lambda_{x_i}\in x_i$ such that
    \[\pi\Lambda_{x_0}\subsetneq\Lambda_{x_1}\subsetneq\Lambda_{x_2}\subsetneq\dots\subsetneq \Lambda_{x_n}\subsetneq\Lambda_{x_0}.\]
\end{itemize}
\end{defn}

As a simplicial complex, $\X(V)$ is isomorphic to the Bruhat--Tits building of $\PGL(V)$.
In particular, $\X(V)$ is connected and pure of dimension $d-1$, where $d=\dim_K(V)$.

\begin{defn}
    Let $\B=\braces{v_1,\dots,v_d}$ be a basis of $V$.
The \emph{apartment} associated to $\B$ is the collection of all vertices represented by lattices of the form
\[\mathrm{Span}_{\O_K}\braces{\pi^{n_1}v_1,\dots,\pi^{n_d}v_d},\]
for integers $n_1,\dots,n_d$.
\end{defn}

Apartments are connected subcomplexes of $\X(V)$, and are pure of dimension $d-1$. Any two apartments are isomorphic as simplicial complexes, and the building $\X(V)$ is covered by its apartments.

\subsection{The invariant complex of a representation}

Now let $\rho\:G\to \GL(V)$ be a representation of a group $G$. Throughout this paper, we assume that there exists at least one lattice in $V$ that is stable under the action of $G$ via $\rho$. This condition is automatically satisfied for most representations of interest. For example, it is satisfied if $G$ is compact and $\rho$ is continuous.

\begin{defn}\label{def:invariant subcomplex}
    We denote by $\X(\rho)$ the subcomplex of $\X(V)$ spanned by the vertices $x\in \X(V)$ that are represented by $\rho(G)$-invariant lattices, i.e.\ such that $\rho(G)\cdot\Lambda_x = \Lambda_x$ for some (and hence every) lattice representing $x$. 
\end{defn}


The Bruhat--Tits building of $\PGL(V)$ has a geometric realisation, in which any apartment becomes a Euclidean space of dimension $d-1$, and there is a well-defined notion of convexity.

\begin{prop}[\cite{Bellaiche-apropos}*{Prop.\ 3.1.3}]\label{thm:convex}
The subcomplex $\X(\rho)$ is convex.
\end{prop} 

There is a unique geodesic between any two points $x, y\in \X(V)$, which we denote by $[xy]$. By \Cref{thm:convex}, if $x, y\in \X(\rho)$ then $[xy]\sub\X(\rho)$.

Many representation-theoretic properties of $\rho$ are encoded in the structure of $\X(\rho)$, in ways that we now recall.

\begin{prop}
    The subcomplex $\X(\rho)$ is bounded if and only if $\rho$ is irreducible.
\end{prop}

\begin{proof}
    The only if direction follows from \cite{Bellaiche-apropos}*{Prop.\ 3.2.1}, which states that $\rho$ is reducible if and only if $\X(\rho)$ contains a half-line.

    When $K$ is locally compact, $\X(\rho)$ is unbounded if and only if it contains a half-line, in which case, the if direction also follows from \cite{Bellaiche-apropos}*{Prop.\ 3.2.1}. For general $K$, the if direction is \cite{suh2023stable}*{Prop.\ 2.2.1}.
\end{proof}

\begin{defn}\label{def:extremal}
    We say that a point $x\in\X(\rho)$ is \emph{extremal} if for every $y,z\in\X(\rho)$, if $x\in[yz]$ then $x = y$ or $x = z$.
\end{defn}

The following result is well-known. A proof is given in \cite{suh2023stable}*{Prop.\ 1.1.1}, where the result is attributed to Serre.

\begin{prop}
Let $x\in \X(\rho)$ and let $\Lambda_x$ be a lattice representing $x$.
Then $x$ is an extremal vertex of $\X(\rho)$ if and only if $\Lambda_x/\pi\Lambda_x$ is indecomposable as a representation of $G$.
\end{prop}

\section{Maximal lattices}\label{sec:max-vertices}

Throughout this section, assume that $(V, \rho)$ is irreducible and admits a stable lattice.

Let $\Lambda \sub V$ be a lattice and let $v\in V$ be a nonzero vector. Then we can always choose a lattice $\Lambda'$, homothetic to $\Lambda$, such that $v\in \Lambda'$ and the image of $v$ in $\Lambda'/\pi\Lambda'$ is nonzero.

\begin{defn}
    We say that $\Lambda$ is \emph{normalised} at $v$ if $v\in \Lambda$ and its image in $\Lambda/\pi\Lambda$ is nonzero. Equivalently, $\Lambda$ is normalised at $v$ if $v\in \Lambda$ but $\pi\ii v\notin \Lambda$.
\end{defn}

Now let $\rho\:G\to \GL(V)$ be a representation, and let $\X(\rho)$ be its invariant building, as defined in \Cref{def:invariant subcomplex}.
For each nonzero vector $v\in V$, we denote by $\LL_v$ the set of all invariant lattices in $V$ that are normalised at $v$.
We equip $\LL_v$ with the partial order given by inclusion.

\begin{defn}
Fix a nonzero vector $v\in V$.
    \begin{enumerate}
    \item We say that a $\rho(G)$-stable lattice $\Lambda$ is maximal with respect to $v$ if $\Lambda$ is maximal in $\LL_v$.
    \item We say that a vertex $x\in \X(\rho)$ \emph{is maximal with respect to $v$} if $x$ is represented by a lattice $\Lambda_x\in\LL_v$ which is maximal with respect to $v$.

    \item We say that a vertex $x\in \X(\rho)$ is \emph{maximal} if it is maximal with respect to some nonzero vector in $V$.
    \end{enumerate}
We denote by $\X_{\max}(\rho)$ the set of all maximal vertices in $\X(\rho)$.
\end{defn} 

\subsection{Existence of maximal vertices}

We first show that, for every nonzero $v\in V$, there exists at least one lattice that is maximal with respect to $v$.

\begin{lem}\label{lem_finite_sequence}
Assume that $(V,\rho)$ is irreducible.
Let $v\in V$ be a nonzero vector.
    \begin{enumerate}
    \item Any strictly increasing sequence of lattices in $\LL_v$ is finite.

    \item There exists a maximal lattice with respect to $v$.
    Moreover, any invariant lattice $\Lambda$ that is normalised at $v$ is contained in a maximal lattice with respect to $v$.
    \end{enumerate}
\end{lem} 

    \begin{proof}
        Part $(ii)$ is a simple consequence of part $(i)$, which we now prove. Let 
    \[\Lambda_0\subsetneq\Lambda_1\subsetneq\Lambda_2\subsetneq\dots\]
    be a strictly increasing sequence of lattices in $\LL_v$, and set $L=\bigcup_{k\geq 0}\Lambda_k$.
    We claim that $L$ is finitely generated as an $\O_K$-module.
    Indeed, it is well known that any $\O_K$-submodule of $V$ is either finitely generated, or contains a $K$-line.
    If $L$ contained a $K$-line, we would have that $L=V$, since $V$ is irreducible, but this is impossible since $\pi^{-1}v\notin L$.
    Therefore, $L$ is a finitely generated module over the Noetherian ring $\O_K$, hence $L$ is a Noetherian module.
    It follows that the ascending chain
    \[\Lambda_0\subsetneq\Lambda_1\subsetneq\Lambda_2\subsetneq\dots\subset L\]
    is finite.
        
    

    \end{proof}

\begin{thm}\label{thm_maximal_equivalence}
Let $v\in V$ be a nonzero vector and let $\Lambda\in\LL_v$ be a $\rho(G)$-invariant lattice that is normalised at $v$.
The following are equivalent.
    \begin{enumerate}
    \item $\Lambda$ is maximal with respect to $v$.
    \item The socle of $\Lambda/\pi\Lambda$ is irreducible and generated by the image of $v$.
    \end{enumerate}
\end{thm} 
    \begin{proof}
    Assume $(i)$. Let $\overline{v}$ be the image of $v$ in $\Lambda/\pi\Lambda$. Then $\overline{v}$ is nonzero by definition. 
    It is sufficient to show that any nonzero subrepresentation of $\Lambda/\pi\Lambda$ contains $\overline{v}$. 
    Indeed, if so, then every irreducible subrepresentation of $\Lambda/\pi\Lambda$ is generated by $\overline{v}$, so $\langle \overline v\rangle$ is the unique irreducible subrepresentation of $\Lambda/\pi\Lambda$.
    
    Let $U\sub \Lambda/\pi\Lambda$ be a subrepresentation not containing $\overline{v}$.
    We will show that $U=0$. 
    Let $\Lambda'$ be the inverse image of $U$ under the quotient map $\Lambda\to \Lambda/\pi\Lambda$. Then $\pi\Lambda\sub\Lambda'\sub\Lambda$ and by definition, $v\notin \Lambda'$. Since $\pi v\in \pi\Lambda\sub\Lambda'$, we see that $v\in \pi\ii\Lambda'\setminus \Lambda'$, so that the lattice $\pi^{-1}\Lambda'$ is normalised at $v$ and contains $\Lambda$.
    By the maximality of $\Lambda$ with respect to $v$, $\pi^{-1}\Lambda'=\Lambda$.
    Therefore, $U=0$.

    Assume $(ii)$.
    Let $\Lambda'$ be an invariant lattice containing $\Lambda$ and normalised at $v$.
    Consider the map 
    \[\varphi\:\Lambda/\pi\Lambda\to \Lambda'/\pi \Lambda'\]
    induced by the inclusion $\Lambda\sub \Lambda'$. We will show that $\varphi$ is an isomorphism, whence $\Lambda'=\Lambda$.

    Since $\varphi$ is a map of $\F$-vector spaces of the same dimension, it is sufficient to show that it is injective. Let $\overline{v}$ denote the image of $v\in\Lambda/\pi\Lambda$. Then $\varphi(\overline v)$ is the image of $v$ in $\Lambda'/\pi\Lambda'$ which is, by assumption, nonzero. Hence $\overline{v}\notin\ker(\varphi)$. But, by assumption, every nonzero subrepresentation of $\Lambda/\pi\Lambda$ contains $\overline{v}$. Hence $\ker(\varphi) = 0$, so $\Lambda = \Lambda'$. So $\Lambda$ is maximal with respect to $v$.
    \end{proof}

\subsection{Sharp subquotients}

Let $\overline{V}$ be a finite-dimensional representation of $G$ over the residue field $\F$.
By a subquotient of $\overline{V}$ we mean a pair $(V_1,V_2)$ of two subrepresentations of $\overline{V}$ such that $V_1\subset V_2$.
We say that a subquotient $(V_1,V_2)$ is irreducible if $V_2/V_1$ is an irreducible representation of $G$.

Fix a $\rho(G)$-invariant lattice $\Lambda$ and let $x\in\X_{\max}(\rho)$ be any maximal vertex. Then we can always choose a representative $\Lambda_x$ of $x$ such that $\Lambda\sub\Lambda_x$ and $\Lambda\not\sub\pi\Lambda_x$, in which case, there is a natural map
\[\varphi\:\Lambda/\pi\Lambda\to\Lambda_x/\pi\Lambda_x\]
induced by the inclusion $\Lambda\sub\Lambda_x$. Now, by \Cref{thm_maximal_equivalence}, the socle of $\Lambda_x/\pi\Lambda_x$ is irreducible. Hence, the pair
\[(\varphi\ii(0), \varphi\ii(\soc(\Lambda_x/\pi\Lambda_x)))\]
is an irreducible subquotient of $\Lambda/\pi\Lambda$.

\begin{defn}\label{def:theta-l}
    For each $\rho(G)$-invariant lattice $\Lambda$, define
    \[\theta_\Lambda\:\X_{\max}(\rho)\to \{\text{irreducible subquotients of }\Lambda/\pi\Lambda\}\]
    by 
    \[x\mapsto (\ker\varphi, \varphi\ii(\soc(\Lambda_x/\pi\Lambda_x))).\]
\end{defn}

In this section, we prove that a subquotient $(V_1, V_2)$ lies in the image of $\theta_\Lambda$ if and only if it is what we call a \emph{sharp subquotient} of $\Lambda/\pi\Lambda$, which we now define.

\begin{defn}
Let $\overline{V}$ be a finite-dimensional representation of $G$ over the residue field $\F$.
We say that a proper subquotient $(V_1,V_2)$ of $\overline{V}$ is \emph{sharp} if, for every subrepresentation $W$ of $\overline{V}$, if $W\supsetneq V_1$, then $W\supseteq V_2$. We let $\operatorname{sharp}(\overline{V})$ denote the set of sharp subquotients of $\overline{V}$.
\end{defn} 

    \begin{exmp}
    Let $V_1$ be a maximal subrepresentation of $\overline{V}$.
    Then $(V_1,\overline{V})$ is a sharp subquotient.
    \end{exmp} 

Clearly, a sharp subquotient is an irreducible subquotient.
The converse is not true in general, however, any irreducible subquotient can be enlarged to a sharp subquotient in the following sense:

\begin{defn}
We say that a sharp subquotient $(V_1',V_2')$ lies above an irreducible subquotient $(V_1,V_2)$ if $V_2\sub V_2'$ and $V_2\cap V_1'=V_1$.
\end{defn}

\begin{prop}\label{prop_sharp_lying_above}
Let $(V_1,V_2)$ be an irreducible subquotient of $\overline V$.
    \begin{enumerate}
    \item There exists a sharp subquotient lying above $(V_1,V_2)$.
    
    \item Let $(V_1',V_2')$ be a sharp subquotient lying above $(V_1,V_2)$.
    Then
    \[V_2'/V_1'\simeq V_2/V_1.\]
    \end{enumerate}
\end{prop} 
    \begin{proof}
        \begin{enumerate}
        \item Let $V_1'$ be a maximal subrepresentation of $\overline V$ containing $V_1$ but not containing $V_2$.
        Denote $V'_2=V_1'+V_2$.
        Then $(V_1',V_2')$ is a sharp subquotient lying above $(V_1,V_2)$.
        
        \item The map $V_2\to V_2'/V_1'$ induced by the  inclusion $V_2\sub V_2'$ is surjective.
        Indeed, the subrepresentation $V_2+V_1'$ strictly contains $V_1'$, so $V_2'\sub V_2+V_1'$.
        The kernel of the map $V_2\to V_2'/V_1'$ is $V_2\cap V_1'=V_1$, hence the map $V_2/V_1\to V_2'/V_1'$ is an isomorphism.     
        \end{enumerate}
    
    \end{proof}

\begin{prop}\label{thm_sharp_bound_length}
Let $c=\mathrm{len}(\overline{V})$ be the length of the $G$-representation $\overline{V}$.
The number of sharp subquotients of $\overline{V}$ is at least $c$.
\end{prop} 
    \begin{proof}
    Let
    \[0=V_0\subsetneq V_1\subsetneq V_2\subsetneq\dots\subsetneq V_{c}=\overline{V}\]
    be a Jordan--Holder sequence for $\overline{V}$.
    Consider the irreducible subquotients $(V_i,V_{i+1})$, for $i=0,\dots,c-1$.
    For each $i$ let $(W_i,U_i)$ be a sharp subquotient lying above $(V_i,V_{i+1})$.
    For each $i$, the subrepresentation $W_i$ contains $V_0,\dots,V_i$ but not $V_{i+1}$.
    Therefore, the spaces $W_0,\dots,W_{c-1}$ are pairwise distinct, so there are at least $c$ sharp subquotients.
    \end{proof} 

\begin{remark}
    Consider the map $\operatorname{sharp}(\overline V)\to \set{\text{Jordan--Holder factors of }\overline V}$ given by $(V_1, V_2)\mapsto V_2/V_1$. We'll see later in \Cref{lem_uniqueness} that this map is surjective. If $\overline{V}$ is multiplicity free, then it is injective, so the number of sharp subquotients of $\overline V$ is at most $c$. Hence, by \Cref{thm_sharp_bound_length}, if $\overline{V}^{ss}$ is multiplicity free, then the number of sharp subquotients of $\overline{V}$ is exactly $c$.

    Indeed, suppose that $(V_1, V_2)$ and $(V_1', V_2')$ are two sharp subquotients such that $V_2/V_1\cong V_2'/V_1' \cong W$, where $W$ is some Jordan--Holder factor of $\overline V$. Since $\overline V$ is multiplicity free, it follows that $W$ is not a Jordan--Holder factor of $V_1$ or of $V_1'$. Now let $W_1 = V_1 + V_1'$. If $W_1\supsetneq V_1$, then since $(V_1, V_2)$ is sharp, we must have $W_1\supset V_2$. But $W$ is a Jordan--Holder factor of $V_2$ but not of $W_1$, a contradiction. Thus $V_1' \subset V_1$. Swapping the roles of $V_1'$ and $V_1$, we see that $V_1' = V_1$. But then $V_2'\supsetneq V_1$, so since $(V_1, V_2)$ is sharp, we have $V_2'\supset V_2$. Since $V_2'/V_1'\cong V_2/V_1'$, it follows that $V_2'= V_2$. Hence, $(V_1, V_2) = (V_1', V_2')$, which proves that the map is injective.

    On the other hand, if $\overline{V}$ is not multiplicity free, then this map need not be injective, and the number of subquotients of $\overline{V}$ can be greater than $c$. For example, if $\overline{V}$ is the trivial representation, then $(W,\overline V)$ is sharp for every codimension-one subrepresentation $W$ of $\overline V$.
\end{remark}

\subsubsection{Sharp subquotients and maximal vertices}

Now suppose that $(\rho, V)$ is a representation of $G$ and that $\Lambda$ is a $\rho(G)$-invariant lattice $\Lambda$, and recall the map
\[\theta_\Lambda\:\X_{\max}(\rho)\to \{\text{irreducible subquotients of }\Lambda/\pi\Lambda\}\]
from \Cref{def:theta-l}.

\begin{lem}
    For every $x\in\X_{\max}(\rho)$, the subquotient $\theta_\Lambda(x) = (\ker(\varphi), \varphi\ii(\soc(\Lambda_x/\pi\Lambda_x)))$ is a sharp subquotient of $\Lambda/\pi\Lambda$.
\end{lem}

\begin{proof}
    By \Cref{thm_maximal_equivalence}, the socle $\soc(\Lambda_x/\pi\Lambda_x)$ is an irreducible representation of $G$. Clearly $V_1 = \ker(\varphi) = \varphi\ii(0)$ is a proper subrepresentation of $V_2 = \varphi\ii(\soc(\Lambda_x/\pi\Lambda_x))$. Moreover, if $W\sub \Lambda/\pi\Lambda$ is such that $V_1\subsetneq W$, then $\varphi(W)$ is a nonzero subrepresentation of $\Lambda_x/\pi\Lambda_x$, and hence $\varphi(W)\supseteq \soc(\Lambda_x/\pi\Lambda_x)$. It follows that
    \[V_2 = \varphi\ii(\soc(\Lambda_x/\pi\Lambda_x))\sub\varphi\ii(\varphi(W)) = W.\]
\end{proof}

It follows that the map $\theta_\Lambda$ can be viewed as a map $\theta_\Lambda\:\X_{\max}(\rho)\to \operatorname{sharp}(\Lambda/\pi\Lambda)$.

\begin{thm}\label{thm_maximal_sharp}
The map $\theta_\Lambda\:\X_{\max}(\rho)\to \operatorname{sharp}(\Lambda/\pi\Lambda)$ is surjective.
\end{thm} 

    \begin{proof}
    Let $(V_1,V_2)$ be a sharp subquotient of $\Lambda/\pi\Lambda$.
    Let $v\in \Lambda$ be a vector whose image in the quotient $\Lambda/\pi\Lambda$ is contained in $V_2$ but not in $V_1$.
    Let $\Lambda'\sub \Lambda$ be the inverse image of $V_1$ under the quotient map $\Lambda\to\Lambda/\pi\Lambda$.
    Then $\Lambda\sub\pi^{-1}\Lambda'$, and both $\Lambda$ and $\pi^{-1}\Lambda'$ are normalised at $v$.
    By \Cref{lem_finite_sequence}, there exists an invariant lattice $\Lambda_x$, maximal with respect to $v$, and containing $\pi^{-1}\Lambda'$.
    Consider the map 
    \[\varphi\:\Lambda/\pi\Lambda\to \Lambda_x/\pi\Lambda_x\] 
    induced by the inclusion $\Lambda\sub\Lambda_x$. By definition, $V_1 = \Lambda'/\pi\Lambda$. Since $\pi\ii\Lambda'\sub \Lambda_x$, it follows that $V_1\sub\ker(\varphi)$. On the other hand, since the image of $v$ is contained in $V_2$ and $\Lambda_x$ is normalised at $v$, $V_2\nsubset\ker(\varphi)$.
    Since $(V_1,V_2)$ is a sharp subquotient, we deduce that $V_1=\ker(\varphi)$. Moreover, since $\varphi^{-1}(\soc(\Lambda_x/\pi\Lambda_x))$ strictly contains $V_1$, it must contain $V_2$. But since $x\in\X_{\max}(\rho)$, by \Cref{thm_maximal_equivalence}, $\soc(\Lambda_x/\pi\Lambda_x)$ is irreducible, so $V_2=\varphi^{-1}(\soc(\Lambda_x/\pi\Lambda_x))$.
    Thus, $\theta_\Lambda(x)=(V_1,V_2)$.    
    \end{proof}



\section{Applications}\label{section_applications}

One of the key applications of invariant lattices, first pioneered by Ribet \cite{Ribet-herbrand}, is to construct interesting representations of the group $G$ over the residue field $\F$. 

Let $\rho\: G\to \GL(V)$ be a representation. Given a stable lattice $\Lambda\sub V$, we can view $\rho$ as a representation $\rho\:G\to \GL(\Lambda)$, and obtain the residual representation $\orho_\Lambda\:G\to \GL(\Lambda/\pi\Lambda)$. In general, the isomorphism class of $\orho_\Lambda$ depends on $\Lambda$, but its semisimplification, $\orho^{ss}$ is independent of all choices. We assume that  $\orho^{ss}=\bigoplus_{i=1}^rW_i^{m_i}$, where the $W_i$'s are distinct, irreducible representations of $G$ over $\F$.

\begin{defn}\label{def:realises}
    Let $\alpha\in \Ext_{\F[G]}(W_i, W_j)$ be a class of non-split extensions of $W_i$ by $W_j$, for some $1\le i,j\le r$.  We say that an invariant lattice $\Lambda\sub V$ \emph{realises} $\alpha$ if there exist subrepresentations $U_0\sub U_1\sub \Lambda/\pi\Lambda$ such that $U_1/U_0$ fits into a short exact sequence
    \[0\to W_j\to U_1/U_0\to W_i\to 0\]
    that defines the extension class $\alpha$.
    We say that \emph{$\rho$ realises $\alpha$} if some invariant lattice $\Lambda\subset V$ realises $\alpha$.
\end{defn}

Let $U$ be a representation of $G$ over $\F$, and suppose that it fits into a short exact sequence
    \begin{equation}\label{eq_short_exact}
    0\to W_j\xmap{f}U\xmap{g}W_i\to 0.
    \end{equation} 

Then we denote by $[U,f,g]$ the extension class in $\Ext_{\F[G]}(W_i,W_j)$ defined by \Cref{eq_short_exact}.
In general, the maps $f$ and $g$ are not unique: the map $f$ is only unique up to the action of $\Aut_G(W_{j})$, and the map $g$ is only unique up to the action of $\Aut_G(W_{i})$.

\begin{defn}\label{def:equivalence}
    We say that two extension classes $[U, f, g]$ and $[U', f', g']$ are \emph{equivalent} if $U\cong U'$ and there are automorphisms $\varphi\in \Aut_G(W_{j})$ and $\psi\in \Aut_G(W_{i})$ such that $f = f'\circ\varphi$ and $g = \psi\circ g'$.
\end{defn}

This notation of equivalence is weaker than the usual notation of equality in $\Ext_{\F[G]}(W_i, W_j)$, where the automorphisms $\phi$ and $\psi$ are required to be the identity maps.

\subsection{Uniqueness of extensions realised by modular representations}


In this subsection, we show that if $U$ is a non-split extension of $W_i$ by $W_j$, for some $1\leq i,j\leq r$, that can be realised by $\rho$, then $U$ can also be realised by a maximal lattice. As a consequence, if $W_i$ and $W_j$ occur in $\orho^{ss}$ with multiplicity one, we prove that $\rho$ can realise at most one non-split extension of $W_i$ by $W_j$, up to equivalence.

\begin{lem}\label{lem_uniqueness}
 Let $1\leq i\leq r$.
     \begin{enumerate}
     \item There exists $x\in \X_{\max}(\rho)$ such that $\soc(\Lambda_x/\pi\Lambda_x)\simeq W_i$.
      \item If $m_i=1$, there is a unique $x\in \X_{\max}(\rho)$ such that $\soc(\Lambda_x/\pi\Lambda_x)\simeq W_i$.
     \end{enumerate}
 \end{lem} 

     \begin{proof}
         $(i)$. Choose any invariant lattice $\Lambda$.
         Since $W_i$ appears in the semisimplification of $\Lambda/\pi\Lambda$, there exists an irreducible subquotient $(V_1,V_2)$ of $\Lambda/\pi\Lambda$ such that $V_2/V_1\simeq W_i$.
         By \Cref{prop_sharp_lying_above}, there exists a sharp subquotient $(V_1',V_2')$ lying above $(V_1,V_2)$, and $V_2'/V_1'\simeq W_i$.
         By \Cref{thm_maximal_sharp}, there exists a maximal vertex $x\in\X_{\max}(\rho)$ such that $\theta_\Lambda(x)=(V_1',V_2')$.
         From the definition of $\theta_\Lambda$ it follows that $\soc(\Lambda_x/\pi\Lambda_x)\simeq V_2'/V_1'\simeq W_i$.

         $(ii)$. Let $\Lambda_1,\Lambda_2$ be two invariant lattices such that $\soc(\Lambda_1/\pi\Lambda_1)\simeq\soc(\Lambda_2/\pi\Lambda_2)\simeq W_i$.
         By scaling, we may assume that  $\Lambda_1\subset \Lambda_2$, and $\Lambda_1\nsubset \pi\Lambda_2$.
         The map
         \[\varphi\:\Lambda_1/\pi\Lambda_1\to \Lambda_2/\pi\Lambda_2.\]
         induced by the inclusion $\Lambda_1\subset \Lambda_2$ is nonzero, because $\Lambda_1\nsubset \pi\Lambda_2$.
         Thus, the image of $\varphi$ contains the socle of $\Lambda_2/\pi\Lambda_2$.
         It follows that $(\Lambda_1/\pi\Lambda_1)/\ker(\varphi)$ contains a subrepresentation isomorphic to $W_i$.
         Since $m_i=1$, the kernel of $\varphi$ does not contain a subrepresentation isomorphic to $W_i$.
         Any non-zero subrepresentation of $\Lambda_1/\pi\Lambda_1$ contains $\soc(\Lambda_1/\pi\Lambda_1)$, since it is the unique irreducible subrepresentation of $\Lambda_|/\pi\Lambda_1$.
         Therefore, $\ker(\varphi)=0$, so $\Lambda_1=\Lambda_2$.
     \end{proof} 

\begin{remark}
Condition $(ii)$ that $m_i = 1$ is necessary: if $m_i>1$, then there could be distinct lattices $\Lambda,\Lambda'$ with $\Lambda/\pi\Lambda\not\cong\Lambda'/\pi\Lambda'$, but with $\soc(\Lambda/\pi\Lambda)\cong\Lambda'/\pi\Lambda'\cong W_i$. For example, let $G$ be the quaternion group of order $8$ and $K=\Q_2(i)$, where $i^2=-1$.
Then $G$ has a unique $2$-dimensional irreducible representation $(\rho,V)$ over $K$.
If $\sigma\in \Aut(G)$, then $\rho\circ \sigma$ is a $2$-dimensional irreducible representation of $G$, hence isomorphic to $\rho$.
Thus, every $\sigma\in\Aut(G)$ induces an operator $T_\sigma\:V\to V$ intertwining the two representations.
It is easy to see that $T_\sigma$ takes a stable lattice to a stable lattice, and moreover takes a maximal lattice to a maximal lattice.
Thus, we have an action of $\Aut(G)\simeq S_4$ on the maximal lattices in $V$.

The residue field of $K$ is $\F_2$, and the only irreducible representation of $G$ over $\F_2$ is the trivial one.
Therefore, if $\Lambda\subset V$ is a maximal stable lattice, $\Lambda/\pi\Lambda$ is a non-split extension of the trivial representation $\F_2$ by itself.
A standard computation shows that
\[\Ext^1_G(\F_2,\F_2) = H^1(G,\F_2)=\Hom(G,\F_2)\simeq \F_2\oplus \F_2,\]
and one can compute the action of $\Aut(G)$ on $\Ext_G^1(\F_2,\F_2)$ and see that it is irreducible (we note that $S_4$ has a unique $2$-dimensional irreducible representation over $\F_2$).

The action of $\Aut(G)$ on lattices is compatible with its action on their mod-$\pi$ reductions and on $\Ext^1_G(\F_2,\F_2)$.
Therefore, there must be two maximal lattices in $V$ whose reductions modulo $\pi$ are not isomorphic.
\end{remark}

\begin{thm}\label{thm_ext_classes_by_maximal_lattices}
    Suppose that $\rho$ realises a non-split extension class $\alpha=[U,f,g] \in \Ext_{\F[G]}(W_i, W_j)$. Then there is a {maximal} lattice $\Lambda_x$ such that $U$ is isomorphic to a subrepresentation of $\Lambda_x/\pi\Lambda_x$.
    In particular, the extension class $\alpha$ is realised by $\Lambda_x$.
\end{thm}

    \begin{proof}
    Let $\Lambda$ be an invariant lattice that realises $\alpha=[U,f,g]$.
    Then there exist subrepresentations $V_1\sub V_2\sub V_3\sub \Lambda/\pi\Lambda$ such that $V_3/V_1\simeq U$, $V_2/V_1\simeq W_{j}$, and $V_3/V_2\simeq W_{i}$.
    The subquotient $(V_1,V_2)$ is irreducible, so by \Cref{prop_sharp_lying_above}, there exists a sharp subquotient $(V_1',V_2')$ lying over $(V_1,V_2)$.
    By definition, $V_1'\cap V_2=V_1$. 
    Moreover, we have $V_1'\cap V_3=V_1$. 
    Indeed, $V_1'\cap V_3$ is a subrepresentation of $V_3$ containing $V_1$.
    Since $V_3/V_1\simeq U$ is a non-split extension of two irreducible
representations, its only subrepresentations are $0$, $W_j$, and $U$.
Thus the only subrepresentations of $V_3$ containing $V_1$ are
$V_1,V_2,$ and $V_3$.
    Since $V_1'$ does not contain $V_2$, we must therefore have $V_1'\cap V_3=V_1$.
    
     
    By \Cref{thm_maximal_sharp}, the map $\theta_\Lambda$ is surjective.
    Thus, there exists a maximal lattice $\Lambda_x$ such that $\theta_\Lambda(x)=(V_1',V_2')$.
    By definition, the kernel of the map
    \[\Lambda/\pi\Lambda\xrightarrow{\eta}\Lambda_x/\pi\Lambda_x\]
    is $V_1'$, and $V_2'$ is the preimage of $\soc(\Lambda_x/\pi\Lambda_x)$.
    Let $U'=\eta(V_3)$, a subrepresentation of $\Lambda_x/\pi\Lambda_x$.
    Then since $V_1'\cap V_3=V_1$, we have $U'\simeq V_3/V_1\simeq U$.
    \end{proof} 

\begin{proof}[Proof of Theorem $\ref{thm:uniqueness}$]
    Suppose that $\rho$ realises two extension classes $[U,f,g]$ and $[U',f',g']\in \Ext_{\F[G]}(W_i, W_j)$. By \Cref{thm_ext_classes_by_maximal_lattices}, there are maximal vertices $x, x'$ such that $U$ is isomorphic to a subrepresentation of $\Lambda_x/\pi\Lambda_x$, and $U'$ is isomorphic to a subrepresentation of $\Lambda_{x'}/\pi\Lambda_{x'}$. By \Cref{thm_maximal_equivalence}, the socles of $\Lambda_x/\pi\Lambda_x$ and $\Lambda_{x'}/\pi\Lambda_{x'}$ are irreducible. Hence, both socles must be isomorphic to $\soc U \simeq \soc U' \simeq W_j$. 
    By \Cref{lem_uniqueness}, and since $m_j=1$, we have $x = x'$.
    Thus, there are subrepresentations $V,V'\subset\Lambda_x/\pi\Lambda_x$ such that $V\simeq U$ and $V'\simeq U'$.

    Now, since $W_i$ occurs with multiplicity one in the Jordan--Holder decomposition of $\Lambda_x/\pi\Lambda_x$, it follows that $V=V'$, hence $U\simeq U'$.
    This means that the classes $[U,f,g]$ and $[U',f',g']$ differ by the action described in \Cref{def:equivalence}.
\end{proof}

\begin{remark}
    By duality, for each $i = 1, \ldots, r$, there exists a lattice $\Lambda$ such that $\mathrm{cosoc}(\Lambda/\pi\Lambda)\simeq W_i$. Indeed, applying \Cref{lem_uniqueness} to the dual representation $\rho\dual$, there is a maximal lattice $L\subset V\dual$ such that $\soc(L/\pi L)\simeq W_i\dual$. We can then take $\Lambda$ to be the dual lattice $L\dual =\{v\in V:f(v)\in\O_K \text{ for all } f\in L\}$ of $L$. See \cite{Thompson}*{Thm.\ 1} for a related result in the case that $G$ is a finite group.
\end{remark}

\subsection{On Bellaïche's generalisation of Ribet's Lemma}

In this section, we prove \Cref{thm:bellaiche-generalisation}. We first define the socle filtration of a representation of $G$ over $\F$.

\begin{defn}\label{def:socle}
    If $\overline V$ is a finite-dimensional representation of $G$ over the residue field $\F$, then the socle filtration of $\overline V$ is the filtration
\[0 =\soc^0(\overline V)\subsetneq\soc^1(\overline V)\subsetneq \soc^2(\overline V)\subsetneq\cdots\subsetneq\soc^{\l}(\overline V) = \overline V,\]
where, for each $i\ge 1$,
\[\frac{\soc^i(\overline V)}{\soc^{i-1}(\overline V)} = \soc\br{\frac{\overline V}{\soc^{i-1}(\overline V)}}.\]
In particular, $\soc^1(\overline V) = \soc(\overline V)$. 
\end{defn}

We say that a subrepresentation $W$ of $\overline V^{ss}$ occurs in the $i$-th level of the socle filtration if $W$ is a subrepresentation of $\frac{\soc^i(\overline V)}{\soc^{i-1}(\overline V)}$.

\begin{lem}\label{lem_socle_filtration}
Let $\overline V$ be a representation of $G$ over $\F$. Let $W$ be a Jordan--Holder factor of $\overline V$ and suppose that $W$ occurs in the $i$-th level of the socle filtration of $\overline V$ for some $i\ge 2$. Then there exists a Jordan--Holder factor $U$ of $\overline V$ occurring in the $(i-1)$-th level of the socle filtration of $\overline V$, and representations
\[\soc^{i-2}(\overline V)\sub V_1\sub V_2\sub \soc^{i}(\overline V)\]
such that $V_2/V_1$ is a non-split extension of $W$ by $U$.

\end{lem} 
    
    \begin{proof}
        Quotienting by $\soc^{i-2}(\overline V)$, we may assume that $i = 2$. Let $V_2$ be the subrepresentation of $\soc^{2}(\overline V)$ such that $V_2/\soc(\overline V)\simeq W$. Write
        \[\soc(\overline V) = \bigoplus_{j=1}^nU_j\]
        for the decomposition of $\soc(\overline V)$ into irreducible subrepresentations and, for each $j = 1, \ldots, n$, let
        \[U_j' = \bigoplus_{k\ne j}U_k.\]
        Since $\bigcap_{j=1}^nU_j' = 0$, the map
        \[V_2\to\bigoplus_{j=1}^nV_2/U_j'\]
        is injective. Since $V_2$ is not semisimple (since otherwise it would be contained in $\soc(\overline V)$), there exists some $j$ such that $V_2/U_j'$ is not semisimple.

        So set $V_1 = U_j'$ and $U=U_j$. Then $0\sub V_1\sub V_2\sub\soc^2(\overline V)$ and there is a non-split short exact sequence
        \[0\to\frac{\soc(\overline V)}{V_1}\to \frac{V_2}{V_1}\to \frac{V_2}{\soc\overline V}\to 0.\]
        By construction, we have $\frac{\soc(\overline V)}{V_1}\simeq U$ and $\frac{V_2}{\soc\overline V}\simeq W$.
        
    \end{proof} 

\begin{proof}[Proof of Theorem $\ref{thm:bellaiche-generalisation}$]
    
    The existence of $\Lambda_i$ is exactly \Cref{lem_uniqueness}$(i)$.
    For the remaining assertion, we proceed by induction on $m$. If $m=1$, then $i=j$ and there is nothing to prove.
    Write $\overline V = \Lambda_{i}/\pi\Lambda_{i}$ and suppose that $W_j$ occurs in the $m$-th level of the socle filtration of $\overline V$ with $m\ge 2$. Then, by \Cref{lem_socle_filtration}, there exists an integer $k$ such that $W_k$ occurs in the $(m-1)$-th level of the socle filtration, and $\rho$ realises a non-split extension of $W_j$ by $W_k$. Hence, there is a directed edge from $W_k$ to $W_j$, and the result follows by induction.       
\end{proof}

\subsection{Geometry of maximal vertices}\label{sec:suh}

In \cite{suh2023stable}, Suh studies the extremal vertices of $\X(\rho)$, and shows two 
abundance results: a lower bound on the number of extremal vertices, and a geometric result that says that $\X(\rho)$ is the convex hull of its extremal vertices.

The goal of this section is to prove stronger versions of these two results by reframing them in terms of maximal lattices. We first note that every maximal lattice is an extremal lattice. We then give a lower bound on the number of maximal lattices, which is already stronger than the lower bound given in \cite{suh2023stable}*{Thm.\ 1.1.2}. Finally, we show that $\X(\rho)$ is the \emph{tropical} convex hull of the set of maximal vertices.

Recall from \Cref{def:extremal} that an extremal vertex of $\X(\rho)$ is a vertex $x$ such that for all $y,z\in \X(\rho)$, if $x$ lies on the geodesic $[yz]$, then $x=y$ or $x=z$. We write $\X_{\ext}(\rho)$ for the set of extremal vertices of $\X(\rho)$.

\begin{prop}\label{prop_maximal_is_extremal}
Any maximal vertex $x\in \X_{\max}(\rho)$ is an extremal vertex of $\X(\rho)$.
\end{prop} 

    \begin{proof}
    Let $x\in \X_{\max}(\rho)$ and let $\Lambda_x$ be a representative lattice.
    By \Cref{thm_maximal_equivalence}, $\Lambda_x/\pi\Lambda_x$ is indecomposable, and it follows from \cite{suh2023stable}*{Proposition 1.1.1} that $x$ is an extremal vertex.
    \end{proof} 

In \cite{suh2023stable}*{Thm.\ 1.1.2}, Suh shows that $|\X_{\ext}(\rho)| \ge \min(3, 1 + \dim(\X(\rho)))$. We prove the following strengthening of this result.

\begin{thm}\label{thm_maximal_vertices_lower_bound}
We have 
\[\abs{\X_{\max}(\rho)}\geq 1+\dim(\X(\rho)).\]
with equality if $\rho$ is residually multiplicity-free. In particular, we have $\abs{\X_{\ext}(\rho)}\ge 1+\dim(\X(\rho))$.
\end{thm} 
    
    \begin{proof}
    By \Cref{prop_maximal_is_extremal}, we have $\abs{\X_{\ext}(\rho)}\geq \abs{\X_{\max}(\rho)}$, so it suffices to prove the first statement.

    By \cite{Bellaiche-apropos}*{Prop.\ 3.3.2}, we have
    \[1+\dim(\X(\rho))=\sum_{i=1}^rm_i,\]
  where $\orho^{ss}=\bigoplus_{i=1}^rW_i^{m_i}$.
  
    Let $\Lambda$ be an invariant lattice.
    Then the length of the representation $\Lambda/\pi\Lambda$ is $\sum_{i=1}^rm_i$.
    By \Cref{thm_maximal_sharp} and \Cref{thm_sharp_bound_length}, we have
    \[\abs{\X_{\max}(\rho)}\geq \abs{\operatorname{sharp}(\Lambda/\pi\Lambda)}\geq \sum_{i=1}^rm_i=1+\dim(\X(\rho)).\]
    Finally, assume that $\rho$ is residually multiplicity-free, i.e.\ that $m_i = 1$ for all $i$.
    Then $1+\dim(\X(\rho))=\sum_{i=1}^rm_i=r$, so by \Cref{lem_uniqueness}, the number of maximal vertices is exactly $r$.
    \end{proof}

Next, we show that the invariant building $\X(\rho)$ is, in a sense, spanned by the maximal vertices, and that the maximal vertices are a minimal spanning set.

\begin{defn}[\cite{Develin2004}]
    We say that a simplicial subcomplex $S\sub\X(V)$ is \emph{tropically convex} if, for every two lattices $\Lambda_1,\Lambda_2$ that represent vertices in $S$, the intersection $\Lambda_1\cap\Lambda_2$ also represents a vertex in $S$.
\end{defn}

\begin{prop}
    Both $\X(V)$ and $\X(\rho)$ are tropically convex.
\end{prop}

\begin{proof}
    The tropical convexity of $\X(V)$ is clear, since the intersection of any two lattices is a lattice. Similarly, if $\Lambda_1, \Lambda_2$ are invariant lattices, then so is $\Lambda_1\cap\Lambda_2$.
\end{proof}

\begin{defn}\label{def:tropical}
    Let $S\sub \X(V)$ be a subset of the vertices of $\X(V)$. The \emph{tropical convex hull} of $S$ is the intersection of all the tropically convex subsets of $\X(V)$ containing $S$.
\end{defn}


\begin{thm}\label{thm:trop-convex}
The tropical convex hull of $\X_{\max}(\rho)$ is $\X(\rho)$.
\end{thm} 

    \begin{proof}
    Let $\Lambda$ be a $\rho(G)$-invariant lattice. We will show that $\Lambda$ is an intersection of finitely many maximal lattices.
    
    For each $x\in \X_{\max}(\rho)$, let $\Lambda_x$ denote the unique lattice representing $x$ such that $\Lambda\sub \Lambda_x$, but $\Lambda\not\sub\pi\Lambda_x$.
    Let $L=\bigcap_{x\in\X_{\max}(\rho)}\Lambda_x$.
    We claim that $\Lambda=L$.
    Clearly, $\Lambda\sub L$.
    Let $v$ be a vector such that $\Lambda$ is normalised at $v$.
    Then $\Lambda\in \LL_v$, and there exists a lattice $\Lambda'\in\LL_v$ containing $\Lambda$ and maximal in $\LL_v$.
    Then $\Lambda'$ represents $x$, for some $x\in\X_{\max}(\rho)$.
    In addition, $\Lambda\sub\Lambda'$ and $\Lambda\not\sub\pi\Lambda'$.
    Thus, $\Lambda'=\Lambda_x$.
    It follows that $L$ is also normalised at $v$.
    Thus, $L$ is normalized at $v$, for every $v$ at which $\Lambda$ is.
    Then the map $\Lambda/\pi\Lambda\to L/\pi L$, induced by the inclusion $\Lambda\sub L$, is injective.
    Therefore, $\Lambda=L$. 

    If the residue field $\F$ is infinite, then the set $\X_{\max}(\rho)$ can also be infinite. 
    Thus, it remains to show that $\Lambda = \bigcap_{x\in\X_{\max}(\rho)}\Lambda_x$ is actually a finite intersection. But for any $x\in \X_{\max}(\rho)$, there exists an integer $n$ such that $\pi^n\Lambda_x\sub\Lambda\sub \Lambda_x$, so $\Lambda_x/\Lambda$ is a finitely generated $\O_K/\pi^n\O_K$-module. In particular, $\Lambda_x/\Lambda$ is Artinian, so there is a finite subset $\{x_1, \ldots, x_n\}\sub\X_{\max}(\rho)$ such that $\Lambda = \bigcap_{i=1}^n\Lambda_{x_i}$.
    \end{proof}

\begin{remark}
In a two dimensional apartment of the Bruhat--Tits building, the tropical convex hull of two points is always one-dimensional, whereas the simplicial convex hull can be two-dimensional.

We expect that the tropical convex hull of a set $S\subset \X(V)$ of vertices is always contained in the simplicial convex hull of $S$.
If that is true, then $\X(\rho)$ is also the simplicial convex hull of $\X_{\max}(\rho)$.
We note that \cite{suh2023stable}*{Thm.\ 1.1.2(1)},  shows that $\X(\rho)$ is the simplicial convex hull of the set of \emph{all} extremal vertices of $\X(\rho)$. Our result suggests that we only need the subset of maximal vertices.
\end{remark}



Finally, we show that $\X_{\max}(\rho)$ is the smallest subset whose tropical convex hull is $\X(\rho)$.

\begin{prop}
    If $S\sub \X(\rho)$ is a subset whose tropical convex hull is $\X(\rho)$, then $\X_{\max}(\rho) \sub S$.
\end{prop} 
    \begin{proof}
    Let $x\in \X_{\max}(\rho)$ and suppose that $x$ is maximal with respect to $v\in V$.
    Let $\Lambda_x$ be a lattice representing $x$ and normalised at $v$.
    Then there exist $s_1,\ldots,s_n\in S$ and lattices $\Lambda_i$ representing $s_i$ such that 
    \[\Lambda_x=\bigcap_{i=1}^n\Lambda_i.\]
    Since $v\in \Lambda_x\setminus\pi\Lambda_x$, there must exist some $1\leq i\leq n$ such that $\Lambda_i$ is normalised at $v$.
    Then $\Lambda_x\sub \Lambda_i$, and since $\Lambda_x$ is maximal with respect to $v$, we have $\Lambda_x=\Lambda_i$.
    Therefore, $x\in S$.
    \end{proof}



\section*{Acknowledgements}

We thank the anonymous referee, Ehud de Shalit, and Junecue Suh for helpful comments and suggestions. AW was supported by an AMS--Simons Travel Grant.

\bibliographystyle{alpha}
\bibliography{bibliography}
\end{document}